\documentclass{article}
\usepackage{amsmath,amsthm,amsfonts}
\usepackage{graphicx}
\usepackage{subfigure}
\usepackage{natbib}
\usepackage[affil-it]{authblk}

\theoremstyle{plain}
\newtheorem{theorem}{Theorem}

\newtheorem{proposition}{Proposition}

\theoremstyle{definition}
\newtheorem{definition}{Definition}

\theoremstyle{remark}
\newtheorem{remark}{Remark}

\newtheorem*{rmk}{Remark on non-dimensionalization}

\begin{document}

\title{Passivity-based PI control \\ of first-order systems \\ with I/O communication delays: \\ A complete $\sigma$-stability analysis}

\author{Fernando Casta\~nos, Edgar Estrada, Sabine Mondi\'e and Adri\'an Ram\'irez}
\affil{Department of Automatic Control, Cinvestav-IPN, \\Av. IPN No. 2508, Col. Zacatenco C.P. 07360, M\'{e}xico, D.F.}

\maketitle

\begin{abstract}
The PI control of first-order linear passive
systems through a delayed communication channel is revisited in light of the relative stability concept called
$\sigma$-stability. Treating the delayed communication channel as a transport PDE, the passivity of the overall control-loop is guaranteed, resulting in a closed-loop system of neutral nature. Spectral methods are then applied to the system to obtain a complete stability map. In particular, we perform the $\mathcal{D}$-subdivision method to declare the exact $\sigma$-stability regions in the space of PI parameters. This framework is then utilized to analytically determine the maximum achievable exponential decay rate $\sigma^{*}_{d}$ of the system while achieving the PI tuning as explicit function of $\sigma^{*}_{d}$ and system parameters.
\end{abstract}

\section{Introduction}

Passivity-based control relies on the fact that the power-preserving
interconnection of two passive subsystems yields again a passive system \citep{ortegaR}.
Lyapunov stability of the interconnected system follows from passivity while
asymptotic stability is usually achieved by adding appropriate damping. Due
to its simplicity and robustness, passivity-based control has attracted
researchers and practitioners in the control community for several
decades, e.g. \citep{youla1959,willems1972a,willems1972b,hill1976,byrnes1991b,schaft}.

However, if a delayed communication channel stands between the plant
and the controller, as in a typical control scenario, passivity arguments fail
due to non-passive properties of the channel. Here, the loss of passivity follows from the fact that the Nyquist plot of a pure delay does not lie in the right-hand side of the complex plane. In their seminal work, inspired from the study of transmission lines~\citet{anderson1989} proposed a useful modification of the communication channel to remedy the aforementioned design problem. More precisely, the communication channel is transformed into a passive system, thus recovering the simplicity and effectiveness of the passivity-based design.
This idea has been discussed in many contributions and has given rise to an
outstanding number of proposals addressing optimality issues, applications in the field of robotics, motor
control, among other studies. The reader is referred to~\citet{nuno2011} for a recent
tutorial.

This paper revisits the modified communication channel from the perspective
of time-delay systems theory using spectral methods considering first-order linear passive systems. The problem is motivated by the wide variety of industrial processes described by first-order plants with time-delay and commonly regulated by PI controllers \citep{silva2001}, such as DC servomotors extensively used in industry. In this paper, the emphasis is put on the performance of the closed-loop system when a communication channel stands between the plant and the controller, which is quantified by its $\sigma$-stability degree. Here, $\sigma$ approximates the exponential decay rate of the system response.

Our analysis is based on classical results of time-delay systems
of retarded and neutral nature~\citep{bellman,haleVerduyn}. The problem under consideration,
though infinite dimensional, involves a reduced number of parameters. Hence, a
comprehensive frequency domain analysis of the closed-loop characteristic
quasipolynomial can be performed. Particularly, we deploy a critical extension of
the $\mathcal{D}$-subdivision method of~\citet{neimark1949},
see also~\citet{sipahi2011} for advanced methods, which consists on (i)
the determination of the stability boundaries corresponding to roots at $s=-\sigma $
and $s=-\sigma +j\omega $, which provides a partition of the space of
PI parameters and (ii) the verification of the relative stability degree $\sigma$ of each region in the partition. Having generated the complete set of $\sigma$-stability boundaries and determined the $\sigma$-stability regions, the exact $\sigma$-stability maps follow. This framework then results in a fully analytic characterization of the maximum achievable exponential decay rate for which simple tuning formulae for practitioners are finally declared.

The contribution is organized as follows: In Section~\ref{sec:delay}, the
delay-free and the fixed, non-zero delay cases are analyzed, illustrating
the failure and loss of performance in the passivity-based design strategy.
In Section~\ref{sec:scattering}, the scattering transformation is introduced
and the characteristic function is obtained, the $\sigma$-stability maps are
sketched and tuning rules for points of interest are derived. A theoretical limit
on the $\sigma$-stability is found and a tuning rule for the scattering
transformation is proposed. It is shown that, when the rule is followed,
the theoretical limit can be approached arbitrarily close. Concluding remarks
are given in Section~\ref{sec:conc}.

\section{Problem statement} \label{sec:delay}

Consider a first-order linear system of the form
\begin{subequations} \label{eq:plant}
 \begin{align}
 \dot{x} &= -ax+bu_{1} \;, \\
   y_{1} &= x \;,
\end{align}
\end{subequations}
where $u_{1}$, $y_{1}$ and $x$ $\in \mathbb{R}$ are the input, output and state,
respectively. The parameters $a$ and $b$ are assumed to be non negative, which ensures
passivity with storage function $V_1(x) = x^2/(2b)$.

Consider the PI controller
\begin{subequations} \label{eq:PI}
\begin{align}
 \dot{\xi} &= u_{0} \;, \\
     y_{0} &= k_{p}u_{0}+k_{i}\xi \;,
\end{align}
\end{subequations}
where $u_{0}$, $y_{0}$, $\xi $ $\in \mathbb{R}$ are the controller input, output and
state, respectively. The proportional and integral gains $k_{p}$ and $k_{i}$
are assumed to be positive, so the controller is also passive. The system and the controller are interconnected as per the following pattern
\begin{equation*}
u_{1}=-y_{0}\quad\textnormal{and}\quad u_{0}=y_{1}-y^{\star}_{1}.
\end{equation*}

Since the interconnection of two passive systems is again passive, the closed-loop
characteristic polynomial is stable for all gains.

\begin{figure}
\begin{center}
 \includegraphics[width=0.80\linewidth]{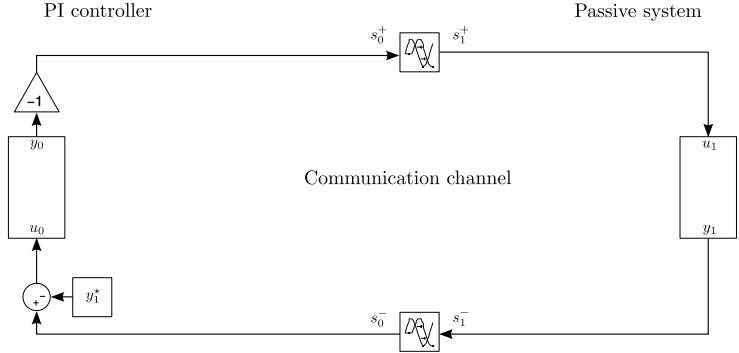}
\end{center}
\caption{A control-loop including a delayed communication channel.}
\label{fig:NoTrans}
\end{figure}

However, when a communication channel with delays is introduced in the loop, as
shown in Fig.~\ref{fig:NoTrans}, the closed-loop transfer function takes the form
\begin{equation} \label{eq:TF}
 \frac{y_{1}(s)}{y_{1}^{\ast }(s)} =
  \frac{(k_{p}s+k_{i})be^{-h_{1}s}}{s^{2}+as+(k_{p}s+k_{i})be^{-hs}}
\end{equation}
with $h_1$ the forward delay, $h_2$ the return delay and $h=h_{1}+h_{2}$ the round-trip
delay. The stability properties of the closed-loop are then defined by the location of the roots of the characteristic equation
\begin{equation} \label{eq:den_hp}
 p(s)=s^{2}+as+(k_{p}s+k_{i})be^{-hs}\;,
\end{equation}
also known as the characteristic quasipolynomial. Notice that the presence of the delay $h>0$ in the communication channel induces infinite-dimensionality to the system due to the exponential term and therefore the quasipolynomial in (\ref{eq:den_hp}) bears an infinite number of roots. Since it is impossible to compute all these roots, the stabilization of the zero-solution is not trivial. Moreover, since the delay channel is not passive, the passivity argument fails and stability can no longer be ensured for every combination of positive parameters.

In the following we consider the problem of finding the setting on the parameters $(k_{p},k_{i})$ that create the maximum decay rate for the system (\ref{eq:plant})-(\ref{eq:PI}) in the presence/absence of time-delays.

\subsection{Performance degradation as a result of delays} \label{sec:delayFree}

As a preliminary step for the characterization of the maximum decay rate, we begin with the decomposition of the $(k_{p},k_{i})$-plane. Besides pure stability, we will be concerned with the exponential decay of system solutions
with a given degree $\sigma$, that is, with the $\sigma$-stability of the system. This
happens only if all the roots of the characteristic quasipolynomial have real parts
less than $-\sigma$~\citep{gu}. Then, we will determine the set of all $(k_{p},k_{i})$ points for which the closed-loop system
is $\sigma$-stable.

Following the $\mathcal{D}$-subdivision method, we equate~\eqref{eq:den_hp} to zero
and set $s=-\sigma$, which gives the boundary
\begin{equation} \label{eq:kikp_hpe}
 k_{i}=\sigma k_{p}+\frac{\sigma a_\sigma}{be^{h\sigma}} \;.
\end{equation}
Now, setting $s=-\sigma +j\omega$ and solving for $k_{p}$ and $k_{i}$ gives the parametric
equations
\begin{subequations} \label{eq:kikp_hpp}
\begin{align}
 k_{p}(\omega ) &= \frac{-a_{2\sigma}\omega\cos(h\omega)+\left(\sigma a_\sigma+\omega^{2}\right)\sin(h\omega)}{b\omega e^{h\sigma}} \;, \label{eq:kp_hpp} \\
 k_{i}(\omega ) &= \frac{(\sigma^2+\omega^{2})\left(\omega\cos(h\omega)+a_\sigma\sin(h\omega)\right)}{b\omega e^{h\sigma}} \;. \label{eq:ki_hpp}
\end{align}
\end{subequations}
Equipped with~\eqref{eq:kikp_hpe},~\eqref{eq:kikp_hpp} and under continuity arguments, we can now declare the exact regions in $(k_{p},k_{i})$-plane for which the quasipolynomial~\eqref{eq:den_hp} is $\sigma$-stable.

\begin{rmk}
 In investigating the spectral properties of the considered control-loop it is not necessary to distinguish between every possible combination of $(a,b)$ parameters. In fact, introducing the following scaling variables $\tau = a t$ and $u_1' = (b/a) u_1$, the non-dimensional form of the general
 plant~\eqref{eq:plant} is obtained as
 \begin{align*}
  \frac{\mathrm{d}}{\mathrm{d} \tau}x(\tau) &= -x(\tau) + u_1'(\tau) \;, \\
                                  y_1(\tau) &= x(\tau) \;,
 \end{align*}
 where the dependence on $(a,b)$ is obviated. In other words, without loss of generality we can set $a=b=1$ in the figures that follow
 for demonstration of the control-loop properties. In consequence, and consistent with the above non-dimensionalization, the figures for all possible $(a,b)$ combinations might change quantitatively, but the qualitative features remain unchanged.
\end{rmk}

Considering that $\sigma$ is given, equations~\eqref{eq:kikp_hpe},~\eqref{eq:kikp_hpp} decompose the $(k_{p},k_{i})$-plane into a finite number of disjoint regions. Due to continuity arguments, each of these regions is then characterized by the same number of strictly $\sigma$-unstable characteristic roots. We will refer to the collection of all pairs $(k_{p},k_{i})$ for which the number of $\sigma$-unstable roots of $p(s)$ is zero as the $\sigma$-stability domain $\mathcal{D}_{\sigma}$.

Using ${a=b=1}$ in ~\eqref{eq:kikp_hpe},~\eqref{eq:kikp_hpp}, we obtain Fig.~\ref{fig:noScat} for $h\geq0$. In order to assist the reader, for a given $\sigma$ the corresponding $\mathcal{D}_{\sigma}$ is filled in color and the stability crossing boundaries are trimmed to fit the boundary of $\mathcal{D}_{\sigma}$. Two interesting observations are in order. Firstly, in Fig.~\ref{fig:regions_hp} one can observe that searching for a faster response will eventually result in the collapse of $\mathcal{D}_{\sigma}$. We have that, at the critical value $\sigma^{*}$ where the boundaries obtained from ~\eqref{eq:kikp_hpe},~\eqref{eq:kikp_hpp} meet each other at $\mathcal{D}_{\sigma^{*}}$, a triple real dominant root will be generated at $-\sigma^{*}$. This point corresponds to the maximum achievable closed-loop exponential decay rate $\sigma^{*}$, and these arguments are consistent with previous discussions on the maximum achievable $\sigma$, see \citep{michiels} Theorem 7.6. Secondly, from Fig.~\ref{fig:regions_h0}, in contrast with the delayed case, it is clear that achieving an arbitrarily
large exponential decay $\sigma$ is possible in the absence of communication delays.

\begin{figure}
\begin{center}
 \subfigure[Delay-free case, $h=0$. Achievable decay is unbounded.]{
  \includegraphics[width=0.8\textwidth]{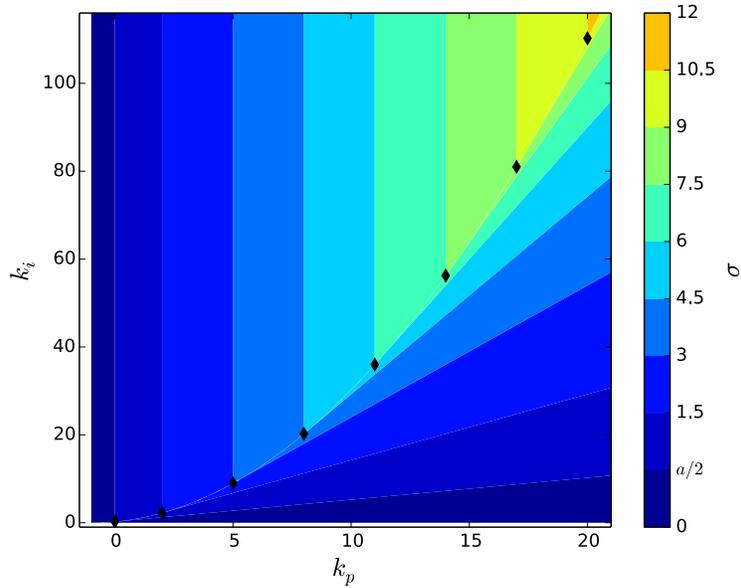}
  \label{fig:regions_h0}
 }
 \subfigure[With delays, $h=0.1$. Maximal achievable decay is $\sigma^\ast = 6.35$.]{
  \includegraphics[width=0.8\textwidth]{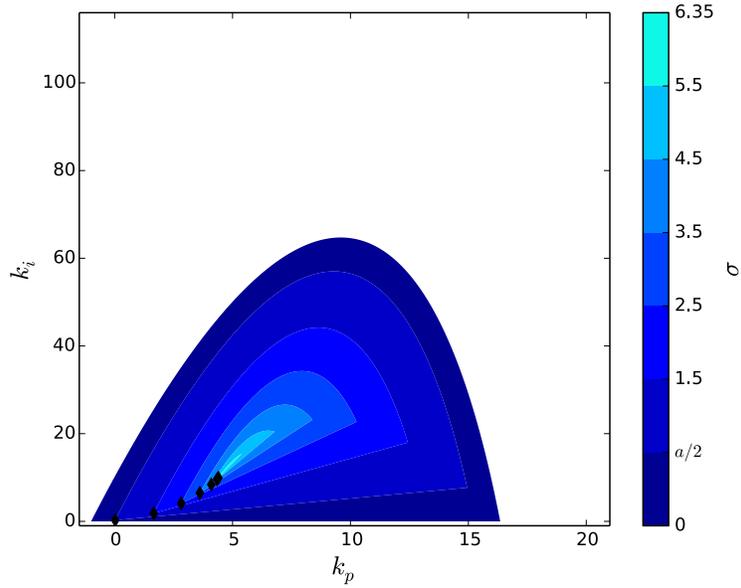}
  \label{fig:regions_hp}
 }
 \caption{$\sigma$-stability regions. Communication channel without scattering transformation.
  The diamond markers correspond to minimal gains ensuring a given $\sigma$.} \label{fig:noScat}
\end{center}
\end{figure}

Let us characterize the maximal achievable decay $\sigma^\ast$ for the delayed case. This characterization is then particularized to conclude on the relative stability properties of the delay-free case.

\begin{proposition} \label{prop:hp}
Consider a plant~\eqref{eq:plant} in closed-loop with a PI controller~\eqref{eq:PI}
satisfying $k_{p} \ge 0$ and $k_{i} \ge 0$. A delayed communication channel with round-trip delay
$h > 0$ stands between the system and the controller.
\begin{enumerate}
 \item[(i)] The maximal achievable exponential decay is given by
  \begin{equation} \label{eq:s*_hp}
   \sigma^\ast=\frac{4+ah-\sqrt{8+a^{2}h^{2}}}{2h} \;.
  \end{equation}
 \item[(ii)] The minimal PI controller gains assigning a given $\sigma^\ast \ge \sigma \ge a/2$
  are
  \begin{subequations}
  \begin{align}
   k_{p} &= \frac{\sigma ha_\sigma-a_{2\sigma}}{be^{h\sigma}} \;, \label{eq:kp_hpmin} \\
   k_{i} &= \frac{\sigma^2(ha_\sigma+1)}{be^{h\sigma}} \;, \label{eq:ki_hpmin}
  \end{align}
  \end{subequations}
  where condition $\sigma \ge a/2$ ensures that $k_p \ge 0$ for all $h > 0$.
\end{enumerate}
\end{proposition}

\begin{proof}
\begin{enumerate}
 \item[(i)] For a given $\sigma$, the corresponding stability domain is delimited by the parametric equations corresponding to a
  real root at $s = -\sigma$ and a pair of complex roots $s = -\sigma \pm j\omega$
  of the quasipolynomial~\eqref{eq:den_hp}. As a consequence, the collapse of the $\sigma$-stability domain occurs at a triple root
  at $-\sigma$. Thus, the quasipolynomial~\eqref{eq:den_hp}, its first and its second derivatives
  must vanish at $s=-\sigma$. That is,
  \begin{subequations} \label{eq:triple_hp}
  \begin{align}
             -b\sigma e^{h\sigma}k_{p}+be^{h\sigma}k_{i} &= \sigma a_\sigma \;, \label{eq:den_hpS} \\
         be^{h\sigma}(1+h\sigma)k_{p}-hbe^{h\sigma}k_{i} &= -a_{2\sigma} \;, \label{eq:den_hpS'} \\
   -hbe^{h\sigma}(2+h\sigma)k_{p}+h^{2}be^{h\sigma}k_{i} &= -2\;.  \label{eq:den_hpS''}
  \end{align}
  \end{subequations}
  Solving these equations for $\sigma$ gives
  \begin{equation*}
   \sigma_{1,2} = \frac{4+ah\pm \sqrt{8+a^{2}h^{2}}}{2h} \;.
  \end{equation*}
  The solution that ensures $k_{p} \ge 0$ and $k_{i} \ge 0$ is~\eqref{eq:s*_hp}.
 \item[(ii)] The minimal gains occur at the intersection of the two boundaries. Hence, the conditions~\eqref{eq:den_hpS}
  and~\eqref{eq:den_hpS'} must hold. The result follows by solving these equations for $k_{p}$ and $k_{i}$.
\end{enumerate}
\end{proof}

We end this section with a comment on the performance degradation as a result of delays. To this end, let $h\rightarrow0$ in (\ref{eq:s*_hp}). It follows that the maximum achievable decay rate tends to infinite as the delay vanishes. From a practical point of view, an upper bound for $\sigma^{*}$ is solely determined by the physical limitations of the considered delay-free system. Moreover the minimal PI controller gains assigning a given $\sigma\geq a/2$ are given by $k_{p}$ in (\ref{eq:kp_hpmin}) and by $k_{i}$ in (\ref{eq:ki_hpmin}) with $h=0$.

Finally the rest of the paper is dedicated to the problems of (i) improving the relative stability of the system when challenged by the presence of communication delays (ii) recovering the passive properties of the overall control-loop and (iii) algebraically designing the controller gains to prescribe a desired exponential decay rate.

\section{The scattering transformation} \label{sec:scattering}

A classical approach to the study of transmission lines consists in applying
a linear transformation on the state variables~\citep{cheng}. By applying such transformation,
the transmission line equations, i.e., the telegrapher's equations, transform into
a pair of uncoupled delay equations, i.e., transport PDE. It is then possible
to understand the dynamics of the transmission line in terms of wave
propagation.

The reverse argument was proposed by~\citet{anderson1989}: Suppose we have
a communication channel consisting of a pair of delays. Apply the inverse
transformation to emulate the behaviour of a transmission line. Since
transmission lines are passive (lossless), the passivity argument is
restored.

More precisely, consider a pair of delays given by the transport PDE, see \citep{krstic} for applications of transport PDE to backstepping design,
\begin{equation}  \label{eq:delayh}
 \begin{pmatrix}
  s_l^+(l,t) \\ s_l^-(l,t)
 \end{pmatrix}
 =
 \begin{pmatrix}
  -h_{1} & 0 \\
       0 & h_{2}
 \end{pmatrix}
 \begin{pmatrix}
  s_t^+(l,t) \\
  s_t^-(l,t)
 \end{pmatrix} \;,
\end{equation}
where $s^{\pm}_{l}$ are the partial derivatives of the scattering variables $s^{\pm}$ with respect to the spatial variable $l\in[0,1]$.
Notice that, at the boundaries, the $s^{\pm}$ solutions satisfy $s^+(1,t) = s^+(0,t-h_{1})$
and $s^-(1,t) = s^-(0,t+h_{2})$. This is the communication channel. Consider now
the linear transformation
\begin{equation}  \label{eq:T}
 \begin{pmatrix}
  \mu(l,t) \\ \upsilon(l,t)
 \end{pmatrix}
 =
 \begin{pmatrix}
  1 & d \\
  1 & -d
 \end{pmatrix}^{-1}
 \begin{pmatrix}
  s^+(l,t) \\
  s^-(l,t)
 \end{pmatrix} \;,
\end{equation}
where $d > 0$ is a design parameter. It follows from~\eqref{eq:T} and~\eqref{eq:delayh}, with
$h=h_{1}+h_{2}$, in general $h_1 \neq h_2$, that
\begin{equation*}  \label{eq:telh}
 \begin{pmatrix}
  \mu_l(l,t) \\
  \upsilon_l(l,t)
 \end{pmatrix}
 = -\frac{h}{2}
 \begin{pmatrix}
    0 & d \\
  1/d & 0
 \end{pmatrix}
 \begin{pmatrix}
  \mu_t(l,t) \\
  \upsilon_t(l,t)
 \end{pmatrix}
 \;,
\end{equation*}
where $(\mu_{l},\nu_{l})$ and $(\mu_{t},\nu_{t})$ are  respectively the spatial and temporal partial derivatives of $(\mu,\nu)$. The above PDE corresponds to the (lossless) telegrapher's equations.

Transformation~\eqref{eq:T} can be enforced at the boundaries
of~\eqref{eq:delayh}. However, it is necessary to be cautious in respecting the
causality of the system: The variables $s^+_0 := s^+(0,t)$, $s^-_1 := s^-(1,t)$,
$\mu_0 := \mu(0,t)$ and $\upsilon_1 := \upsilon(1,t)$ are free, while
$s^+_1 := s^+(1,t)$, $s^-_0 := s^-(0,t)$, $\upsilon_0 := \upsilon(0,t)$
and $\mu_1 := \mu_1(1,t)$ depend on the former variables and their past
values. With these restrictions in mind, we can set, at the boundary $l = 0$,
\begin{subequations} \label{eq:Ts}
\begin{equation}
 \begin{pmatrix}
  \upsilon_0 \\ s^+_0
 \end{pmatrix}
 =
 \begin{pmatrix}
  1/d & -1/d \\
  2 & -1
 \end{pmatrix}
 \begin{pmatrix}
  \mu_0 \\ s^-_0
 \end{pmatrix} \;,
\end{equation}
(cf.~\eqref{eq:T} at $l = 0$). At $l = 1$ we have
\begin{equation}
 \begin{pmatrix}
  \mu_1 \\ s^-_1
 \end{pmatrix}
 =
 \begin{pmatrix}
  -d & 1 \\
 -2d & 1
 \end{pmatrix}
 \begin{pmatrix}
  \upsilon_1 \\ s^+_1
 \end{pmatrix} \;,
\end{equation}
\end{subequations}
(cf.~\eqref{eq:T} at $l = 1$). Here, equation (\ref{eq:Ts}) is known as the scattering transformation. Further details can be found
in~\citep{nuno2011,nuno2009,niemeyer2004,niemeyer1991}.
Finally, we can use the interconnection pattern
\begin{equation}  \label{eq:intD}
 u_0 = \upsilon_0 - y_1^\star \;, \quad \mu_0 = -y_0 \;, \quad u_1 = \mu_1
  \quad \text{and} \quad \upsilon_1 = y_1 \;,
\end{equation}
as in Fig.\ref{fig:ControlScheme}.

\begin{figure}
\begin{center}
 \includegraphics[width=0.80\linewidth]{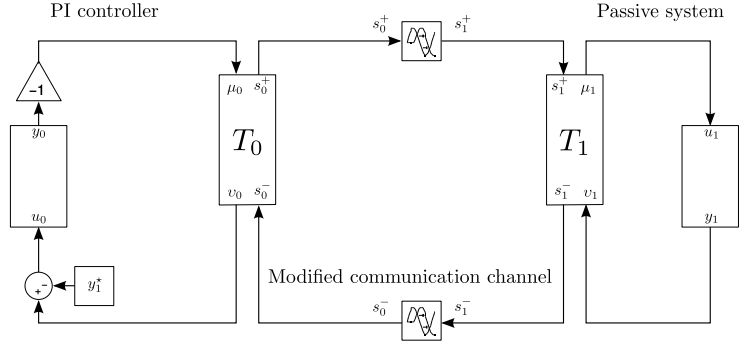}
\end{center}
\caption{A control-loop including a communication channel with the scattering transformation. }
\label{fig:ControlScheme}
\end{figure}

\subsection{Recovering stability and improving $\sigma$-stability}

When the scattering transformations are introduced as shown in Fig.~\ref{fig:ControlScheme},
and the parameter $d$ is arbitrary, the closed-loop transfer function is obtained from (\ref{eq:plant}), (\ref{eq:PI}) and (\ref{eq:Ts}) and using both the interconnection pattern (\ref{eq:intD}) and the equations of the delayed communication channel $s_{1}^{+}=s_{0}^{+}(t-h_{1})$ and $s_{0}^{-}=s_{1}^{-}(t-h_{2})$. Then, we have
\begin{displaymath}
 \frac{y_{1}(s)}{y_{1}^{\ast}(s)} =
  \frac{2d(k_{p}s+k_{i})be^{-h_{1}s}}{p_2(s)s^{2} + p_1(s)s + p_0(s)} \;,
\end{displaymath}
where
\begin{align*}
 p_2(s) &= (1+e^{-hs})d + (1-e^{-hs})k_{p} \;, \\
 p_1(s) &= (1+e^{-hs})(bk_p+a)d + (1-e^{-hs})(bd^2+ak_p+k_i) \;, \\
 p_0(s) &= (1+e^{-hs})bk_i d + (1-e^{-hs}) a k_i \;.
\end{align*}

The characteristic quasipolynomial is thus
\begin{equation}  \label{eq:den_scat}
 p(s) = p_2(s)s^{2} + p_1(s)s + p_0(s) \;.
\end{equation}
Notice that this quasipolynomial is of neutral type. In other words, its principal
coefficient, $p_2(s)$, contains exponential terms, see~\citep{bellman,kharitonov2001} for a discussion and analysis of
quasipolynomials. In the time domain, this means that the time derivative of the
state of the system does not only depend on delayed states, but also on the
derivative of the delayed state.

To determine the boundaries of the $\sigma$-stability region, we equate the
quasipolynomial~\eqref{eq:den_scat} to zero, set $s=-\sigma$ and solve for $k_i$ as a
function of $k_p$,
\begin{equation} \label{eq:kikp_scater_e}
 k_{i} = \sigma k_p + \sigma \frac{(1+e^{hs})a_\sigma + (1-e^{hs})bd}{(1-e^{hs})a_\sigma + (1+e^{hs})bd}d \;.
\end{equation}
Now, setting $s=-\sigma+j\omega$ and solving for $k_{p}$ and $k_{i}$  gives the implicit
parametric equations
\begin{equation} \label{eq:kikp_scater_p}
 A(\omega)
 \begin{pmatrix}
  k_{p} \\
  k_{i}
 \end{pmatrix}
 =
 d\cdot B(\omega) \;,
\end{equation}
where
\begin{align*}
 A_{11}(\omega) &= (\gamma\omega - \alpha\sigma)bd - \beta(\sigma a_\sigma + \omega^2) - \gamma\omega a_{2\sigma} \;, \\
 A_{12}(\omega) &= \alpha bd + \beta a_\sigma - \gamma w \;, \\
 A_{21}(\omega) &= (\gamma \sigma + \alpha \omega)bd - \gamma (\sigma a_\sigma + \omega^2) + \beta \omega a_{2\sigma} \;, \\
 A_{22}(\omega) &= -\gamma bd + \gamma a_\sigma + \beta \omega \;, \\
 B_1(\omega) &= (\beta \sigma+\gamma \omega)bd + \alpha (\sigma a_{\sigma} + \omega^2) - \gamma \omega a_{2\sigma} \;, \\
 B_2(\omega) &= (\gamma \sigma-\beta \omega)bd - \gamma (\sigma a_{\sigma} + \omega^2) - \alpha \omega a_{2\sigma} \;.
\end{align*}
and
\begin{align*}
 \alpha &= 1 + e^{h\sigma}\cos(h\omega) \;, \\
 \beta  &= 1 - e^{h\sigma}\cos(h\omega) \;, \\
 \gamma &= e^{h\sigma}\sin(h\omega) \;.
\end{align*}

Recall that a necessary condition for the stability of neutral-type delay systems is the
stability of the difference operator~\citep{haleVerduyn,michiels,olgac2004,olgac2008}, which is the
inverse Laplace transform of $p_2(s)$. When $\sigma=0$, the difference equation becomes
\begin{displaymath}
 \left( d+k_{p}\right) x(t)+\left( d-k_{p}\right) x(t-h) = 0 \;.
\end{displaymath}
Since
$\left\vert \left( d-k_{p}\right) /\left( d+k_{p}\right) \right\vert <1$, the
stability of the difference operator always holds. On the contrary, when $\sigma > 0$,
the characteristic equation of the difference operator becomes
$d+k_{p}+e^{-hs}e^{h\sigma}(d-k_{p})=0$. The necessity on the stability of the
difference operator imposes the new condition
\begin{displaymath}
 e^{h\sigma}\left|\frac{d-k_{p}}{d+k_{p}}\right| < 1
\end{displaymath}
or, equivalently,
\begin{equation} \label{eq:diff_op}
 d\frac{e^{h\sigma}-1}{e^{h\sigma}+1} < k_p < d\frac{e^{h\sigma}+1}{e^{h\sigma}-1} \;.
\end{equation}

The expressions~\eqref{eq:kikp_scater_e},~\eqref{eq:kikp_scater_p} and~\eqref{eq:diff_op}
are used to determine the $\sigma $-stability regions in the $(k_{p},k_{i},d)$-space of
parameters. These regions are shown in Fig.~\ref{fig:regions_3d}. For illustration
purposes, the slice $d = 15$ is also shown in Fig.~\ref{fig:regions_d15}.

\begin{remark}
The scattering transformation recovers the striking property, observed in the delay-free
case, that the whole first quadrant of the parameter space ensures stability of the closed-loop. Regarding $\sigma$-stability, the size of the regions in Fig.~\ref{fig:regions_d15} are
larger than those in Fig.~\ref{fig:regions_hp}. Also, the maximal achievable exponential decay,
$\sigma^\ast$, is greater with the scattering transformation than without it for
sufficiently large $d$.
\end{remark}

\begin{figure}
\begin{center}
 \subfigure[Regions of $\sigma$-stability in the $(k_{p},k_{i},d)$-space of parameters.
   The upper parts of the regions have been clipped so that inner regions corresponding
   to higher $\sigma$ are visible. The planes $d = k_p$ and $d = \zeta_{\min} k_p$ are
   shown in light gray.]{
  \includegraphics[width=0.8\textwidth]{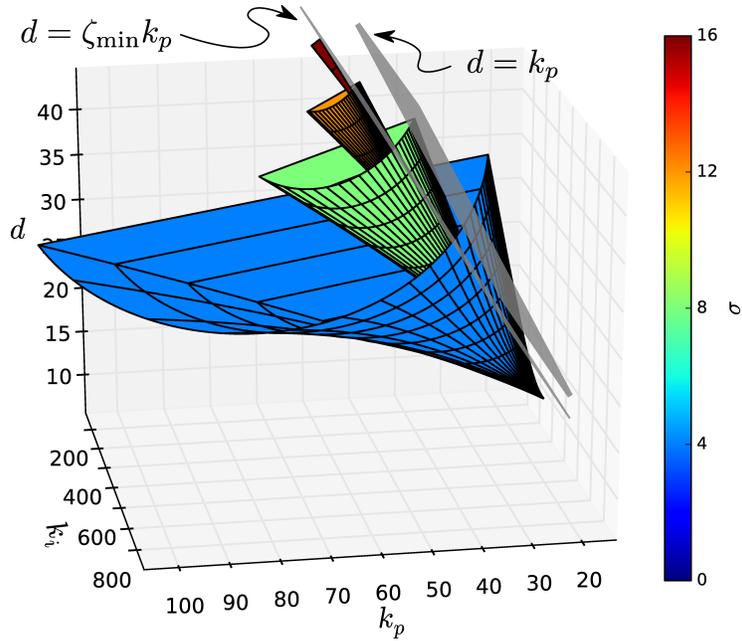}
  \label{fig:regions_3d}
 }
 \subfigure[$\sigma$-stability regions on the slice $d = 15$. The diamond markers correspond
   to the minimal gains that ensure a given $\sigma$. Maximal achievable decay is $\sigma^\ast_{d=15} = 10.9$.]{
  \includegraphics[width=0.8\textwidth]{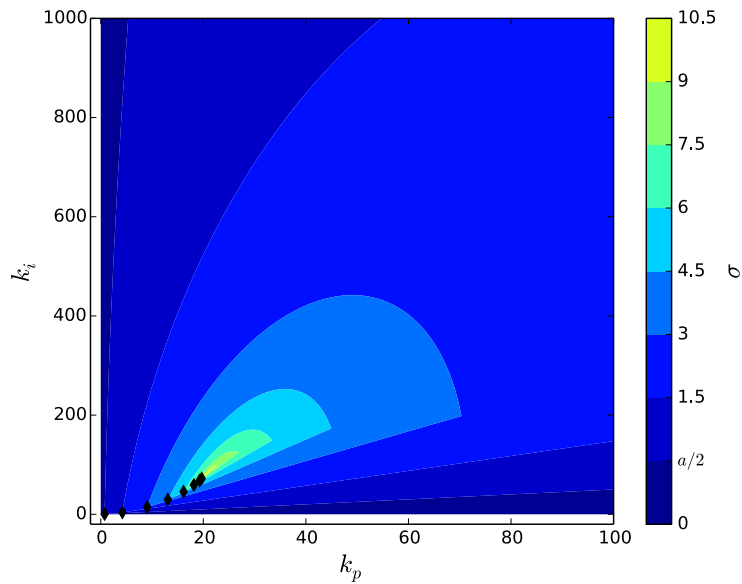}
  \label{fig:regions_d15}
 }
 \caption{Communication channel with delays, $h = 0.1$, and scattering transformation.} \label{fig:Scat}
\end{center}
\end{figure}

Again, we characterize the maximal achievable decay and the corresponding control gains.

\begin{proposition} \label{prop:scat_d}
Consider a plant~\eqref{eq:plant} in closed-loop with a PI controller~\eqref{eq:PI}
satisfying $k_{p} \ge 0$ and $k_{i} \ge 0$. A communication channel with round-trip delay
$h > 0$ stands between the system and the controller. For a given $d > 0$, apply the
scattering transformations~\eqref{eq:Ts} at the channel end points.
\begin{enumerate}
 \item[(i)] The maximal achievable decay $\sigma_{d}^{\ast}$ is a root $\sigma$ of
  $m_{d}(\sigma)$, where
  \begin{multline} \label{eq:s*_scat}
   m_d(\sigma) = (1+e^{h\sigma})\left(2hb^3d^3 + (h^2\sigma a_\sigma+4)b^2d^2 + 2h(\sigma^2-a_{2\sigma}^2)bd -
    h^2\sigma a_\sigma^3 \right) \\
    + (1-e^{h\sigma})\left( h^2\sigma b^3 d^3 + 2h(a+\sigma)b^2d^2 + (4a-h^2\sigma a_\sigma^2)bd - 2ha_\sigma^3 \right) \;
  \end{multline}
  such that the gains
  \begin{subequations} \label{eq:kikp_scat_smax}
  \begin{align}
   k_{p} &= d\frac{(bd-a_\sigma)^{2}e^{2h\sigma} + 2\sigma e^{h\sigma}\left(2bd+h(b^2d^2-a_\sigma^2)\right) - (bd+a_\sigma)^{2}}
    {\left(bd+a_\sigma+e^{h\sigma}(bd-a_\sigma)\right)^{2}} \;, \label{eq:kp_scat_smax} \\
   k_{i} &= 2d\sigma^2e^{h\sigma}\frac{2bd+h(b^2d^2-a_\sigma^2)}
    {\left(bd+a_\sigma+e^{h\sigma}(bd-a_\sigma)\right)^{2}} \;. \label{eq:ki_scat_smax}
  \end{align}
  \end{subequations}
  are non negative and such that~\eqref{eq:diff_op} holds.
 \item[(ii)] The minimal PI controller gains assigning a given $\sigma^\ast_d \ge \sigma > a/2$ are given by~\eqref{eq:kikp_scat_smax}.
\end{enumerate}
\end{proposition}

\begin{proof}
\begin{enumerate}
 \item[(i)] As in the case where no scattering transformation is employed, the
  $\sigma$-stability regions collapse at a triple root at $s=-\sigma$.
  The fact that the quasipolynomial~\eqref{eq:den_scat} and its first derivative
  are zero give the conditions, written in compact form,
  \begin{equation} \label{eq:double_scat}
   \begin{pmatrix}
    A_{11} & A_{12} \\
    A_{21} & A_{22}
   \end{pmatrix}
   \begin{pmatrix}
    k_p \\ k_i
   \end{pmatrix}
   = d\cdot
   \begin{pmatrix}
    B_1 \\ B_2
   \end{pmatrix} \;,
  \end{equation}
  where
  \begin{align*}
   A_{11} &= -(1+e^{h\sigma})\sigma bd - (1-e^{h\sigma})\sigma a_\sigma  \;, \\
   A_{12} &= (1+e^{h\sigma})bd + (1-e^{h\sigma}) a_\sigma \;, \\
   A_{21} &= bd+a_{2\sigma}+e^{h\sigma}\left((bd-a_\sigma)(h\sigma+1)+\sigma\right) \;, \\
   A_{22} &= 1-e^{h\sigma}(h(bd-a_\sigma)+1) \;, \\
  \end{align*}
  and
  \begin{align*}
   B_1 &= (1-e^{h\sigma})\sigma bd + (1+e^{h\sigma})\sigma a_\sigma \;, \\
   B_2 &= -(bd+a_{2\sigma}) + e^{h\sigma}\left((bd-a_\sigma)(h\sigma+1)+\sigma\right) \;. \\
  \end{align*}
  The solution of this linear system of equations is given by~\eqref{eq:kikp_scat_smax}.
  Computing the second derivative of~\eqref{eq:den_scat} gives
  \begin{equation} \label{eq:second_scat}
   \begin{pmatrix}
    A_{31} & A_{32}
   \end{pmatrix}
   \begin{pmatrix}
    k_p \\ k_i
   \end{pmatrix} = d \cdot B_3 \;,
  \end{equation}
  where
  \begin{align*}
   A_{31} &= \left(2-e^{h\sigma}\left((h(bd-a_\sigma)+1)(h\sigma+2)+h\sigma\right)\right) \;, \\
   A_{32} &= he^{h\sigma}(h(bd-a_\sigma)+2) \;, \\
      B_3 &= -\left(e^{h\sigma}\left((h(bd-a_\sigma)+1)(h\sigma+2)+h\sigma\right)+2\right) \;.
  \end{align*}
  Substituting~\eqref{eq:kikp_scat_smax} into~\eqref{eq:second_scat} gives the implicit
  equation $m_d(\sigma) = 0$.
  Then, $\sigma_{d}^{\ast}$ is given as a root of $m_{d}(\sigma)$ satisfying the
  non negative condition on the PI control gains and the necessary stability condition~\eqref{eq:diff_op}
  associated to the neutral nature of the quasipolynomial.
 \item[(ii)] The minimal control gains occur at the intersection of two boundaries. The result follows by noting that at this intersection a double root of the quasipolynomial~\eqref{eq:den_scat} arises, in other words, when~\eqref{eq:kikp_scat_smax} holds.
\end{enumerate}
\end{proof}

Observe that the vanishing of the quasipolynomial~\eqref{eq:den_scat} and
its first two derivatives are necessary conditions for the $\sigma$-stability regions to
collapse. Therefore, the roots of~\eqref{eq:s*_scat} must be verified via back substitution
into the solution of equation~\eqref{eq:double_scat}. Then, if $k_{p} \ge 0$, $k_{i} \ge 0$
and~\eqref{eq:diff_op} hold, the $\sigma_{d}^{\ast}$-stability of the detected collapse point
has to be checked for $\sigma_{d}^{\ast}$ to be feasible.

\begin{remark}
 It follows from the implicit function theorem that $\sigma^\ast_d$ can also be
 characterized as the value of $\sigma$ at which the derivative of~\eqref{eq:kp_scat_smax}
 with respect to $\sigma$ is equal to zero.
\end{remark}

\subsection{Least upper bound on the exponential decay rate}
To provide an idea of how restrictive~\eqref{eq:diff_op} is, we have computed the minimal
gains~\eqref{eq:kikp_scat_smax} and plotted them in Fig.~\ref{fig:kp_min}. Notice that,
when the minimal gains are used, the upper bound in~\eqref{eq:diff_op} is never infringed.
The restriction $\sigma > a/2$, on the other hand, ensures the lower bound.

It is clear from Fig.~\ref{fig:kp_min} that, as $d \to \infty$, the maximal exponential
decays $\sigma^\ast_d$ accumulate at a point, which we denote by $\sigma_{\sup}$. This is
formalized in the following proposition.

\begin{figure}
\begin{center}
 \includegraphics[width=0.80\linewidth]{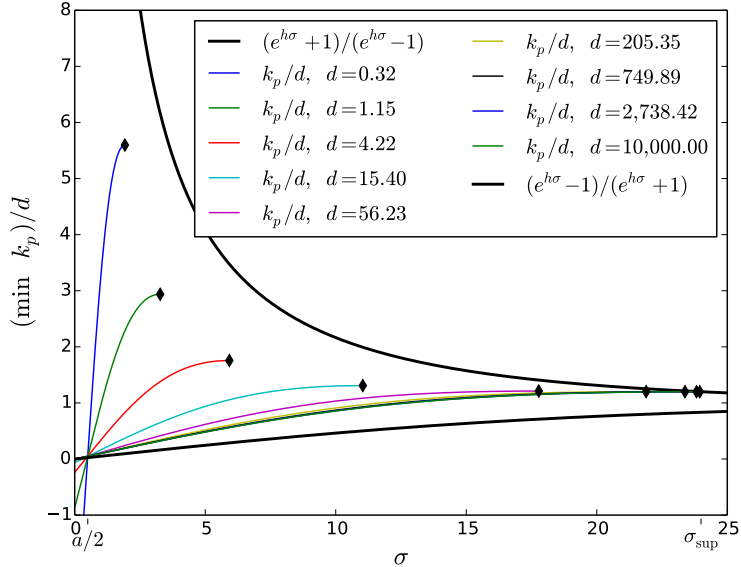}
\end{center}
\caption{Minimal $k_p$ as a function of $\sigma$~\eqref{eq:kp_scat_smax}, normalized by $d$.
 The bounds~\eqref{eq:diff_op} are also shown. The diamond markers correspond to $\sigma^\ast_d$,
 in accordance with~\eqref{eq:s*_scat}. Notice that $k_p(\sigma^\ast_d)/d$ converges to a constant
 value as $d \to \infty$.}
\label{fig:kp_min}
\end{figure}

\begin{proposition}
 The maximal exponential decays $\sigma^\ast_d$ are bounded by $\sigma_{\sup}$, which is defined
 implicitly by
 \begin{equation} \label{eq:s_sup}
  2(1+e^{h\sigma_{\sup}}) + h\sigma_{\sup} (1-e^{h\sigma_{\sup}}) = 0 \;.
 \end{equation}

\end{proposition}

\begin{proof}
 Notice that $m_d(\sigma)$ is asymptotically equivalent to
 \begin{displaymath}
  \left[2(1+e^{h\sigma}) + h\sigma (1-e^{h\sigma})\right]hb^3d^3
 \end{displaymath}
 as $d \to \infty$, the roots of which are given by~\eqref{eq:s_sup}.
\end{proof}

 Notice that $\sigma_{\sup}$ is independent of the system parameters: it depends on $h$ only. In our example we have $h = 0.1$, which gives $\sigma_{\sup} = 23.99$ and agrees with Fig.~\ref{fig:kp_min}.

The notion of \emph{impedance matching}, from the theory of transmission lines, suggests the choice
$d = k_p$ in the scattering transformation~\eqref{eq:Ts}. The rationale behind this choice is that,
in a real transmission line, it avoids wave reflections~\citep{niemeyer2004}. From a frequency-domain
perspective, the immediate advantage of this choice is that the characteristic
quasipolynomial~\eqref{eq:den_scat} simplifies substantially, as the coefficients become
\begin{align*}
 p_2(s) &= 2k_{p} \;, \\
 p_1(s) &= 2k_p(a+bk_p)+(1-e^{-hs})k_i \;, \\
 p_0(s) &= (1+e^{-hs})bk_p k_i  + (1-e^{-hs}) a k_i \;.
\end{align*}
The principal coefficient becomes constant, i.e., the closed-loop
system is of retarded nature instead of neutral, which obviates the need to verify the stability of the difference operator given by the additional constraint~\eqref{eq:diff_op}.

The restriction $d = k_p$ corresponds to the plane shown in Fig.~\ref{fig:regions_3d}. Notice that,
while it is a reasonable choice because of the arguments given above, the plane fails to intersect many
$\sigma$-stability surfaces. Thus, the choice $d = k_p$ is not optimal in the sense that it obstructs
the achievement of $\sigma_{\sup}$-stability. However, notice from Fig.~\ref{fig:kp_min}
that the normalized minimal gain, $k_p(\sigma^\ast_d)/d$, converges to a constant value as $d \to \infty$.
This fact suggests the more general linear relation
\begin{equation} \label{eq:z}
 d = \zeta k_p \;,
\end{equation}
where $\zeta > 0$ is a new design parameter. In the reminder of this section, we will show that there exists
a privileged value of $\zeta$, which we call $\zeta_{\min}$ and that optimizes the bound on $\sigma$.

Let us first compute the $\sigma$-stability boundaries on the plane given by~\eqref{eq:z}. The boundaries
associated to the real roots $s = -\sigma$ are simply given by~\eqref{eq:kikp_scater_e} and~\eqref{eq:z}.
On the other hand, solving~\eqref{eq:kikp_scater_p} with $d = \zeta k_p$ is slightly more difficult, since
$A(\omega)$ and $B(\omega)$ now depend on $k_p$. It follows from Cramer's rule that
\begin{equation} \label{eq:cram}
 \left|A(\omega)\right| k_p + d \left|
  \begin{pmatrix}
   A_2(\omega) & B(\omega)
  \end{pmatrix} \right| = 0 \;,
\end{equation}
where $A_2(\omega)$ stands for the second column of $A(\omega)$ and $|\cdot|$ stands for the
determinant. Substituting~\eqref{eq:z} into~\eqref{eq:cram} gives the equation
\begin{displaymath}
 \left|A(\omega)\right| + \zeta \left|\begin{pmatrix} A_2(\omega) & B(\omega) \end{pmatrix} \right| = 0 \;.
\end{displaymath}
Developing this equation explicitly leads to the second order polynomial equation
$c_2 b^2 \zeta^2 k_p^2 + c_1 b \zeta k_p + c_0 = 0$ with
\begin{align*}
 c_2 &= \left( (\zeta-1)e^{2h\sigma}-(\zeta+1)-2e^{h\sigma}\cos(h\omega) \right)\omega + 2\zeta e^{h\sigma}\sin(h\omega)\sigma \;, \\
 c_1 &= 2\left( \left( -(\zeta-1)e^{2h\sigma} - (\zeta + 1) \right)a_\sigma+2e^{h\sigma}\left( \zeta\sigma\cos(h\omega) + \omega\sin(h\sigma)\right) \right)\omega \;, \\
 c_0 &= \left( \omega^2 + a_\sigma^2 \right)\big( \left( e^{2h\sigma}(\zeta-1)-(\zeta+1) \right)\omega + 2e^{2h\sigma} \left(\omega\cos(h\omega)-\zeta\sigma\sin(h\omega) \right) \big) \;.
\end{align*}
The roots of this polynomial can be computed explicitly. Finally, $k_i$ can be computed as
\begin{displaymath}
 k_i = A_{12}^{-1}(\omega)\left( dB_1(\omega) - A_{11}(\omega)k_p \right)
\end{displaymath}
with~\eqref{eq:z}.

Fig.~\ref{fig:regions_d_kp} shows the $\sigma$-stability regions for the plane
$d = k_p$. Notice that, as in Fig.~\ref{fig:regions_h0}, the boundaries are not given by closed curves.
This is a clear advantage with respect to the case in which $d$ is constant (cf. Fig.~\ref{fig:regions_d15}).
Unfortunately, the choice $d = \zeta k_p$ still imposes a bound on the exponential decay, as it will
be shown shortly.

\begin{figure}
\begin{center}
 \subfigure[With $\zeta = 1$. The exponential decay is bounded by $\sigma^\ast_{\zeta = 1} = 12.78$.]{
  \includegraphics[width=0.8\textwidth]{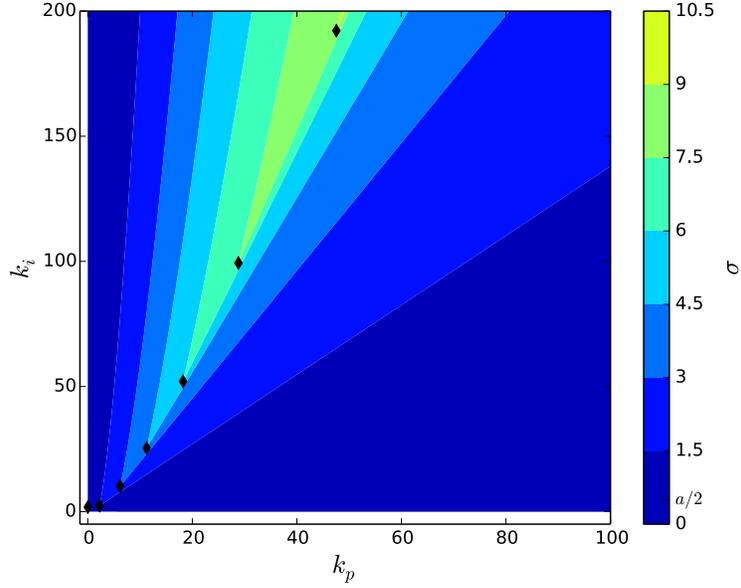}
  \label{fig:regions_d_kp}
 }
 \subfigure[With $\zeta = \zeta_{\min}$. The exponential decay is bounded by
   $\sigma^\ast_{\zeta = \zeta_{\min}} = \sigma_{\sup} = 23.99$.]{
   \includegraphics[width=0.8\textwidth]{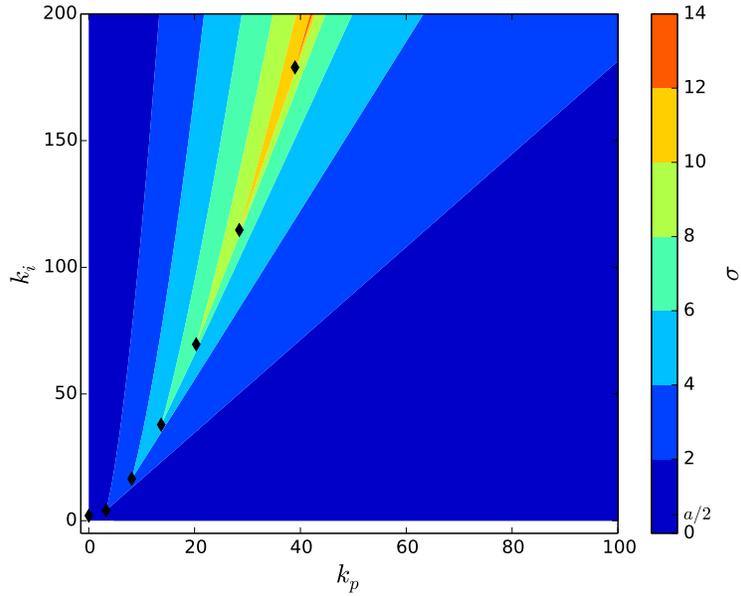}
   \label{fig:regions_d_zkp}
 }
\caption{$\sigma$-stability regions on the plane $d = \zeta k_p$. Communication channel with delay,
 $h = 0.1$, and scattering transformation. The diamond markers correspond to the minimal gains
 that ensure a given $\sigma$.} \label{fig:scatZ}
\end{center}
\end{figure}

\begin{definition}
 Let $\eta_{\sup}$ be the positive solution of
 \begin{equation} \label{eq:eta_min}
  2(1+e^{\eta_{\sup}})+\eta_{\sup}(1-e^{\eta_{\sup}}) = 0
 \end{equation}
 and let $\zeta_{\min}$ be defined as
 \begin{equation*} \label{eq:zeta_eta}
  \zeta_{\min} = \frac{(1+e^{\eta_{\sup}})^2}{2\eta_{\sup} e^{\eta_{\sup}}-(1+e^{\eta_{\sup}})(1-e^{\eta_{\sup}})} \;.
 \end{equation*}
\end{definition}

\begin{proposition} \label{prop:scat_z}
Consider a plant~\eqref{eq:plant} in closed-loop with a PI controller~\eqref{eq:PI}
satisfying $k_{p} \ge 0$ and $k_{i} \ge 0$. A communication channel with round-trip delay
$h > 0$ stands between the system and the controller. Set $d = \zeta k_p$ with $\zeta > 0$
and apply the scattering transformations~\eqref{eq:Ts} at the channel end points.
\begin{enumerate}
 \item[(i)] The least upper bound on the exponential decay is
  \begin{equation*} \label{eq:s*_z}
   \sigma^\ast_\zeta =
    \begin{cases}
     \frac{1}{h}\ln\left(\frac{1+\zeta}{1-\zeta}\right) & \text{if} \quad 0 < \zeta < \zeta_{\min} \\
     \sigma > 0 \text{ such that } m_\zeta(h\sigma) = 0 & \text{if} \quad \zeta_{\min} \le \zeta
    \end{cases} \;,
  \end{equation*}
  where
  \begin{equation} \label{eq:m_z}
   m_\zeta(\eta) = (1+e^{\eta})\left( \zeta(1-e^{\eta}) + 1+e^{\eta} \right)-2\zeta \eta e^{\eta}\;.
  \end{equation}
 \item[(ii)] The minimal $k_p$ assigning a given $\sigma^\ast_\zeta > \sigma > a/2 $ is a root of the
  second order polynomial $c_2 b^2 \zeta^2 k_p^2 + c_1 b \zeta k_p + c_0$ with
  \begin{align*}
   c_2 &= (1+e^{h\sigma})\left( \zeta(1-e^{h\sigma}) + 1+e^{h\sigma} \right)-2\zeta h\sigma e^{h\sigma} \;, \\
   c_1 &= 2(1+e^{h\sigma})\left( \zeta(1+e^{h\sigma}) + 1-e^{h\sigma} \right)a_\sigma-4\zeta a e^{h\sigma} \;, \\
   c_0 &= \left( (1-e^{h\sigma})\left( \zeta(1+e^{h\sigma}) + 1-e^{h\sigma} \right) + 2\zeta h\sigma e^{h\sigma} \right)a_{\sigma}^2 \;,
  \end{align*}
  while the minimal $k_i$ is given by
  \begin{displaymath}
   k_i = \frac{(1+e^{h\sigma})\zeta(bk_p+a_{\sigma}) + (1-e^{h\sigma})(\zeta^2bk_p+a_{\sigma})}{(1+e^{h\sigma})\zeta b k_p+(1-e^{h\sigma})a_\sigma}\sigma k_p \;.
  \end{displaymath}
\end{enumerate}
\end{proposition}

\begin{remark}
 Unlike the bounds given in Propositions~\ref{prop:hp} and~\ref{prop:scat_d}, the bound
 $\sigma^\ast_\zeta$ is not achievable in practice, as it requires infinite gains. It can, however,
 be approached arbitrarily close.
\end{remark}

\begin{proof}
The coefficients of~\eqref{eq:den_scat} take the form
\begin{align*}
 p_2(s) &= (1+e^{-hs})\zeta k_p + (1-e^{-hs})k_{p} \;, \\
 p_1(s) &= (1+e^{-hs})(bk_p+a)\zeta k_p + (1-e^{-hs})(b\zeta^2k_p^2+ak_p+k_i) \;, \\
 p_0(s) &= (1+e^{-hs})bk_i \zeta k_p + (1-e^{-hs}) a k_i \;.
\end{align*}
The formulas for computing the minimal gains are obtained by setting~\eqref{eq:den_scat} and its
derivative equal to zero. This proves (ii).

The maximal exponential decay approaches its maximal value, $\sigma^\ast_\zeta(\zeta)$, as $k_p$ and $k_i$ tend
to infinity, that is, as $c_2$ approaches 0. Notice that $c_2 = 0$ is equivalent to $m_\zeta(h\sigma) = 0$.
For fixed $\zeta$, there exists an $\eta > 0$ such that $m_\zeta(\eta) = 0$ if, and only if,
\begin{displaymath}
 \zeta_{\min} \le \zeta \;.
\end{displaymath}
To see this, set $m_\zeta(\eta)=0$ and write $\zeta$ as a function of $\eta$:
\begin{equation} \label{eq:zeta_eta_2}
 \zeta = \frac{(1+e^{\eta})^2}{2\eta e^{\eta}-(1+e^{\eta})(1-e^{\eta})} \;.
\end{equation}
To find the lower bound we solve $m_\zeta(\eta) = 0$ and
\begin{displaymath}
 \frac{\mathrm{d}}{\mathrm{d} \eta}m_{\zeta}(\eta) = -2e^\eta\left[ \zeta(e^\eta+\eta+1)) + 1+e^\eta \right] = 0
\end{displaymath}
simultaneously for $\eta$. This gives the implicit equation~\eqref{eq:eta_min}. The lower bound on $\zeta$
is finally found by substituting $\eta_{\sup}$ in~\eqref{eq:zeta_eta_2}, that is, it is given by $\zeta_{\min}$.

Thus, $m_\zeta(h\sigma) = 0$ is solvable only if $\zeta_{\min} \le \zeta$. For $0 < \zeta < \zeta_{\min}$,
condition~\eqref{eq:diff_op} determines the maximal exponential decay. When $d = \zeta k_p$,
condition~\eqref{eq:diff_op} reads
\begin{displaymath}
 h\sigma < \ln\left( \left| \frac{1+\zeta}{1-\zeta} \right| \right) \;.
\end{displaymath}
This proves (i).
\end{proof}

The following theorem is important from a practical point of view, since it gives an objective choice for
the free parameter in the scattering transformation.

\begin{theorem} \label{thm:main}
 Consider a plant~\eqref{eq:plant} in closed-loop with a PI controller~\eqref{eq:PI}
 satisfying $k_{p} \ge 0$ and $k_{i} \ge 0$. A communication channel with round-trip delay
 $h > 0$ stands between the system and the controller. Set $d = \zeta_{\min} k_p$ and apply
 the scattering transformations~\eqref{eq:Ts} at the channel end points. 
 The least upper bound on $\sigma$, $\sigma^\ast_{\zeta=\zeta_{\min}}$, is equal to $\sigma_{\sup}$.
\end{theorem}

\begin{proof}
 To compute $\zeta_{\min}$ we first solve~\eqref{eq:eta_min}, which is equivalent to~\eqref{eq:s_sup} with
 $\eta = h\sigma_{\sup}$.
\end{proof}

\begin{figure}
\begin{center}
 \includegraphics[width=0.80\linewidth]{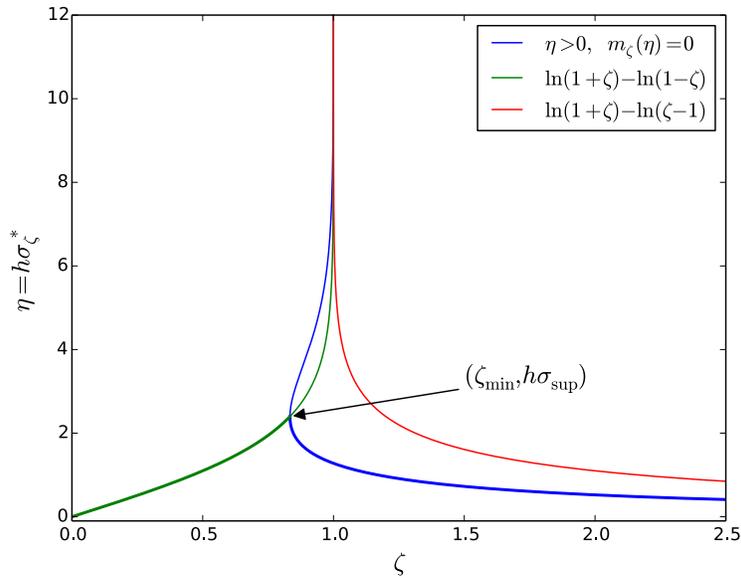}
\end{center}
\caption{Pairs $(\eta,\zeta)$ solution of $m_\zeta(\eta) = 0$, $\eta > 0$ (in blue). Bounds
 given by~\eqref{eq:diff_op} (in green and red). Thick lines correspond to the least upper bound on
 the exponential decay, multiplied by $h$.}
\label{fig:zeta}
\end{figure}

\begin{remark}
 The theoretical limit on the exponential decay is independent of the plant parameters. It
 depends only on the delay $h$ and it is given by
 \begin{displaymath}
  \sigma_{\sup} = \frac{\eta_{\sup}}{h} = \frac{2.3994}{h} \;.
 \end{displaymath}
 The optimal parameter $\zeta_{\min} = 0.8336$, for which $\sigma$ can be set arbitrarily close
 to $\sigma_{\sup}$, is independent of the plant parameters and the delay.
\end{remark}

In order to assist the reader, the necessary stability condition of the difference operator is plotted in Fig.~\ref{fig:zeta} with $d=\zeta k_{p}$ along with the solutions of equation $m_{\zeta}(\eta)=0$ in (\ref{eq:m_z}). As it results from the above analysis, ${\zeta}$ is indeed bounded from below by $\zeta_{min}$. Following this observation, the stability of the difference operator is always satisfied in the interval  $\eta_{\sup}/h > \sigma > 0$ and therefore, the $\sigma$-stability of the overall control-loop is finally established.

The $\sigma$-stability regions on the plane $d = \zeta_{\min} k_p$ are shown in Fig.~\ref{fig:regions_d_zkp}
(compare with Fig.~\ref{fig:regions_d_kp}). Notice that $\sigma^\ast_{\zeta=\zeta_{\min}}$ is almost twice as large
as $\sigma^\ast_{\zeta=1}$, the value obtained using the recipe usually found in the literature. The plane is also shown
in Fig.~\ref{fig:regions_3d}.

In view of the preceding remarks we propose the following design procedure:
\begin{enumerate}
 \item[\textbf{(i)}] Set $d = \zeta_{\min} k_p$. This choice is universal, in the sense that it does not depend explicitly
  on the plant parameters nor the delay.
 \item[\textbf{(ii)}] Choose a desired $\sigma$ such that $\eta_{\sup}/h > \sigma > 0$. This choice requires knowledge of
  the round-trip delay only. The choice of $\sigma$ should take into account the limits imposed by
  the actuator, i.e., large $\sigma$ may generate saturation.
 \item[\textbf{(iii)}] Use the corresponding minimal gains, as in Proposition~\ref{prop:scat_z}. This choice requires knowledge
  of the round-trip delay and the plant parameters.
\end{enumerate}

\section{Concluding Remarks} \label{sec:conc}

A simple instance of the use of the scattering transformation in passivity-based control has been analyzed
from the classical (i.e., frequency-domain) perspective of time-delay systems. The exponential decay rate
of the closed-loop system was chosen as a criterion to asses the performance of different control schemes.
This criterion leads to an optimal choice on the design parameter of the scattering transformation. Quite
remarkably, the optimal choice is independent of the plant parameters and the delays in the
communication channel. With the optimal choice, it is
possible to attain exponential decay ratios that are almost twice as large as those obtained by setting the
design parameter at the value suggested in the literature.

A theoretical limit on the exponential decay rate has been found. The limit does depend on the total delay
but, quite remarkably as well, it is independent of the plant parameters.

Distinct qualitative features, such as the shape and extension of the $\sigma$-stability regions or the neutral
and retarded nature, have been identified for different scenarios: a closed-loop system without delays, with delays,
with and without scattering transformation. This furthers our insight on the effect of delays and the scattering
transformation used to remedy it.

\bibliographystyle{plainnat}
\bibliography{delay}

\end{document}